\documentclass[a4paper,12pt]{article} %{article}
\title{Quasi-compactness of N\'eron models, and an application to torsion points}

\usepackage{amsmath, amssymb, mathrsfs, amsthm, geometry, mathtools, graphicx}% amsfonts, url,, tikz, wrapfig, newclude, }
\usepackage{hyperref}
\usepackage[all]{xy}

\usepackage{enumitem}

\usepackage{cleveref}
\let\oref\ref
\AtBeginDocument{\renewcommand{\ref}[1]{\cref{#1}}}

\newcommand{\on}[1]{\operatorname{#1}}
\newcommand{\bb}[1]{{\mathbb{#1}}}

\newcommand{\ca}[1]{{\mathcal{#1}}}

\newcommand{\abs}[1]{\lvert#1\rvert}

\newcommand{\ra}{\rightarrow}

\newcommand{\hra}{\hookrightarrow}
\newcommand{\sub}{\subseteq}
\newcommand{\tra}{\rightarrowtail}

\theoremstyle{definition}
\newtheorem{definition}{Definition}[section]
\newtheorem{conjecture}[definition]{Conjecture}

\theoremstyle{plain}% default

\newtheorem{lemma}[definition]{Lemma}
\newtheorem{theorem}[definition]{Theorem}
\newtheorem{corollary}[definition]{Corollary}

\theoremstyle{remark}

\renewcommand{\phi}{\varphi}

\author{David Holmes}
\date{\today}

\newcounter{nootje}
\newcommand{\beq}{\begin{equation}}
\newcommand{\eeq}{\end{equation}}
\newcommand{\beqs}{\begin{equation*}}
\newcommand{\eeqs}{\end{equation*}}

\begin{document}
\maketitle
\begin{abstract} 
We prove that N\'eron models of jacobians of generically-smooth nodal curves over bases of arbitrary dimension are quasi-compact (hence of finite type) whenever they exist. We give a simple application to the orders of torsion subgroups of jacobians over number fields. 
\end{abstract}

%In 1977 Mazur gave a bound on the number of rational torsion points on an elliptic curve, but the corresponding conjecture for abelian varieties remains open. Our first main result is to show that this conjecture is equivalent to a natural (but conjectural) generalisation of a theorem of Silverman and Tate on heights in families of abelian varieties. The second part of the paper is devoted to proving some cases of this new conjecture, using the theory of algebraic height jumping. Along the way we prove that N\'eron models of jacobians of proper nodal generically smooth curves are of finite type whenever they exist, which may be of some independent interest. 

%\tableofcontents

\section{Introduction}

If $S$ is a regular scheme, $U \sub S$ is dense open, and $A/U$ is an abelian scheme, then a N\'eron model for $A/S$ can be defined by exactly the same universal property as in the case where $S$ has dimension 1 (we do not impose a-priori that the N\'eron model should be of finite type). Replacing schemes by algebraic spaces for flexibility, we investigate this in detail in \cite{Holmes2014Neron-models-an}, giving necessary and sufficient conditions for the existence of N\'eron models in the case of jacobians of nodal curves. In this short note we prove that, in the setting of jacobians of nodal curves, a N\'eron model is quasi-compact (and hence of finite type) over $S$ whenever one exists. 

Note that N\'eron models of non-proper algebraic groups (such as $\bb{G}_m$) need not be quasi-compact. Moreover, in \cite[\S 10.1, 11]{Bosch1990Neron-models} an example (due to Oesterl\'e) is given of a N\'eron model over a dedekind scheme all of whose fibres are quasi-compact, but which is not itself quasi-compact. Thus in general the question of quasi-compactness of N\'eron models can be somewhat delicate. 

We give an application to controlling the orders of torsion points on abelian varieties. Recall that by \cite{Cadoret2013Note-on-torsion} the uniform boundedness conjecture for jacobians of curves is equivalent to the same conjecture for all abelian varieties. By considering the universal case, the uniform boundedness conjecture is equivalent to
\begin{conjecture}[Uniform boundedness conjecture]\label{conj:UBC_reformulation_easy}
Let $U$ be a scheme over $\bb{Q}$ and $C/U$ a smooth proper curve. Let $\sigma \in J(U)$ be a section of the jacobian. Let $d\ge 1$ be an integer. Then there exists an integer $B$ such that for every point $u \in U(\bar{\bb{Q}})$ with $[\kappa(u):\bb{Q}]\le d$, the order of the point $\sigma(u) \in J_u$ is either infinite or at most $B$. 
\end{conjecture}
We will show that this conjecture holds if there exists a compactification $U \hra S$ over which $C$ has a proper regular model with at-worst nodal singularities, and over which $J$ has a N\'eron model. The first two conditions are relatively mild, especially since it is actually enough to work up to alterations (cf. \cite{Jong1996Smoothness-semi}), but the assumption that a N\'eron model should exist is very strong. For the sake of those readers unfamiliar with the theory developed in \cite{Holmes2014Neron-models-an}, we mention
\begin{corollary}\label{cor_intro:smooth_boundary}
Let $S/\bb{Q}$ be a smooth proper variety, let $Z \tra S$ be a closed subvariety which is smooth over $\bb{Q}$. Let $C/S$ be a family of nodal curves which is smooth outside $Z$. Write $U = S \setminus Z$ and let $\sigma \in J(U)$ be any section of the family of jacobians. Then for every $d \ge 1$ there exists an integer $B$ such that for every point $u \in U(\bar{\bb{Q}})$ with $[\kappa(u):\bb{Q}]\le d$, the order of the point $\sigma(u) \in J_u$ is either infinite or at most $B$. 
\end{corollary}

\noindent More examples are given in \ref{examples}. 

This result is not very useful for proving that the full uniform boundedness conjecture holds for a given family of jacobians; a N\'eron model exists for the tautological curve over a family of curves $C/S$ if and only if $C/S$ is of compact type, in which case the uniform boundedness conjecture is in any case easy (cf. \ref{thm:UBC_for_good_compactifiation}).

Many thanks to Owen Biesel, Maarten Derickx, Bas Edixhoven, Wojciech Gajda, Ariyan Javanpeykar, Robin de Jong and Pierre Parent for helpful comments and discussions. The author is also very grateful to an anonymous referee at Crelle for pointing out a gaping hole in an earlier proof given in \cite{Holmes2014Neron-models-an} of the quasi-compactness of such N\'eron models. 

%Particular thanks to Owen and Robin for allowing me to include the material in \ref{sec:aligned_jump_vanishes} after we removed it from our joint paper \cite{David-Holmes2014Neron-models-an} to avoid circular references. 

\section{The N\'eron model is of finite type}
\label{sec:NM_finite_type}
\newcommand{\Picn}[1]{\on{Pic}^{\abs{-}\le #1}}
\newcommand{\Picz}{\on{Pic}^{[0]}}

 Before giving the proof of our main theorem we briefly recall some definitions we need from \cite{Holmes2014Neron-models-an} and \cite{Holmes2014A-Neron-model-o}. In what follows, $S$ is a scheme. Details can be found in the above references. 

A \emph{nodal curve} over $S$ is a proper flat finitely presented morphism all of whose geometric fibres are reduced, connected, of dimension 1, and have at worst ordinary double point singularities. If $C/S$ is a nodal curve then we write $\Picz_{C/S}$ for the subspace of $\on{Pic}_{C/S}$ consisting of line bundles which have total degree zero on every fibre. If $s\in S$ is a point then a \emph{non-degenerate trait through $s$} is a morphism $f\colon T \ra S$ from the spectrum $T$ of a discrete valuation ring, sending the closed point of $T$ to $s$, and such that $f^*C$ is smooth over the generic point of $T$. 

We say a nodal curve $C/S$ is \emph{quasisplit} if the morphism $\on{Sing}(C/S) \ra S$ is an immersion Zariski-locally on the source (for example, a disjoint union of closed immersions), and if for every field-valued fibre $C_k$ of $C/S$, every irreducible component of $C_k$ is geometrically irreducible. 

Suppose we are given $C/S$ a quasisplit nodal curve and $s \in S$ a point. Then we write $\Gamma_s$ for the \emph{dual graph} of $C/S$ - this makes sense because $C/S$ is quasisplit and so all the singular points are rational points, and all the irreducible components are geometrically irreducible. Assume that $C/S$ is smooth over a schematically dense open of $S$. If we are also given a non-smooth point $c$ in the fibre over $s$, then there exists an element $\alpha \in \ca{O}_{S,s}$ and an isomorphism of completed \'etale local rings (after choosing compatible geometric points lying over $c$ and $s$)
\begin{equation*}
\widehat{\ca{O}^{et}}_{C,c} \stackrel{\sim}{\ra} \frac{\widehat{\ca{O}^{et}}_{S,s}[[x,y]]}{(xy-\alpha)}. 
\end{equation*}
This element $\alpha$ is not unique, but the ideal it generates in $\ca{O}_{S,s}$ is unique. We label the edge of the graph $\Gamma_s$ corresponding to $c$ with the ideal $\alpha\ca{O}_{S,s}$. In this way the edges of $\Gamma_s$ can be labelled by principal ideals of $\ca{O}_{S,s}$. 

If $\eta$ is another point of $S$ with $s \in \overline{\{\eta\}}$ then we get a \emph{specialisation map} 
\begin{equation*}
\on{sp}\colon\Gamma_s \ra \Gamma_\eta
\end{equation*}
on the dual graphs, which contracts exactly those edges in $\Gamma_s$ whose labels generate the unit ideal in $\ca{O}_{S, \eta}$. If an edge $e$ of $\Gamma_s$ has label $\ell$, then the label on the corresponding edge of $\Gamma_\eta$ is given by $\ell\ca{O}_{S,\eta}$.

\begin{theorem}\label{thm:fin_type}
Let $S$ be a regular excellent scheme, $U \hra S$ a dense open subscheme, and $C \ra S$ a nodal curve which is smooth over $U$ and with $C$ regular. Write $J$ for the jacobian of $C_U \ra U$. Assume that $J$ admits a N\'eron model $N$ over $S$. Then $N$ is of finite type over $S$. 
\end{theorem}
\begin{proof} The N\'eron model is smooth and hence locally of finite type; we need to prove that it is quasi-compact. 
%\leavevmode
\begin{itemize}
\item[Step 1:] Preliminary reductions. 

The existence of the N\'eron model implies by \cite{Holmes2014Neron-models-an} that the N\'eron model coincides canonically with the quotient of $\on{Pic}^{[0]}_{C/S}$ by the closure of the unit section $\bar{e}$, and the latter is flat (even \'etale) over $S$. 

 Given an integer $n \in \bb{N}$, we write $\Picn{n}_{C/S}$ for the subfunctor of $\Picz_{C/S}$ consisting of line bundles which on every fibre of $C/S$ have all partial degrees bounded in absolute value by $n$. Clearly $\Picn{n}_{C/S}$ is of finite type over $S$ for every $n$, and we have a composite map $\Picn{n}_{C/S}\ra \Picz_{C/S} \ra N$. If we can show that this composite is surjective for some $n$ then we are done. 

We may assume without loss of generality that $S$ is noetherian. Then we are done if we can make a constructible function $\underline{n}\colon S \ra \bb{N}$ such that for all $s \in S$, the composite $\Picn{\underline{n}(s)}_{C_s/s} \ra N_s$ is surjective. 

Fix a point $x \in S$. We are done if we can find a non-empty open subset $V \hra \overline{\{x\}}$ and an integer $n \in \bb{N}$ such that for all $s \in V$, the composite $\Picn{n}_{C_s/s} \ra N_s$ is surjective. 

\item[Step 2:] Graphs and test curves. 

After perhaps replacing $S$ by an \'etale cover, we may assume that $C/S$ is quasisplit. By \cite[lemma 6.3]{Holmes2014A-Neron-model-o} there exists an open subset $V \hra \overline{\{x\}}$ such that for all $s \in V$, the specialisation map $\on{sp}\colon\Gamma_s\ra\Gamma_x$ is an isomorphism on the underlying graphs. 

\textbf{Claim:} \emph{After shrinking $V$, there exists an integer $m$ such that for all $s \in V$ there exists a non-degenerate trait $f_s\colon T_s \ra S$ through $s$ such that for every edge $e$ of $\Gamma_s$, we have 
\begin{equation*}
\abs{\on{ord}_{T_s}f_s^*\on{label}(e)} \le m. 
\end{equation*}
}
Two things remain to complete the proof: we must prove the claim, and deduce the theorem from the claim. 

\item[Step 3:] Proving the claim. 

Let $f\colon T \ra S$ be a map from a trait to $S$ sending the closed point to $x$ and the generic point to a point in $U$ (this exists by a special case of \cite[7.1.9]{Grothendieck1961EGAII}). 

Write $\bar{T}$ for the schematic image of $T$ in $S$, and $\tilde{T}$ for the normalisation of $\bar{T}$ in $T$. The valuation $\on{ord}_T$ makes sense on elements of $\ca{O}_{\tilde{T}}(\tilde{T})$. After perhaps shrinking $\tilde{T}$ we can assume it is affine (write $R \coloneqq \ca{O}_{\tilde{T}}(\tilde{T})$), and can choose an element $r \in R$ such that $r$ maps to a uniformiser in $T$. 

For each edge $e$ of $\Gamma$, write $\on{label}(e) \in \ca{O}_{S,x}$ for its label. After shrinking $S$, we may assume that $S$ is affine and all the labels lie in $\ca{O}_S(S)$. The pullbacks of the labels to $\tilde{T}$ lie in $\ca{O}_{\tilde{T}}(\tilde{T})$, and so after perhaps shrinking $\tilde{T}$ we can assume that every label is equal (up to multiplication by a unit in $R$) to some power of the `uniformiser' $r$. We are then done by \ref{lem:trait_lemma}. 

\item[Step 4:] Deducing the theorem from the claim.  

We know the formation of $\Picz_{C/S}$ commutes with base change. As remarked above, $\bar{e}$ is flat over $S$ by our assumption that the N\'eron model exists. Because of this flatness, it follows that formation of the closure of the unit section also commutes with base-change. 

Now let us fix for each $s \in V$ a non-degenerate trait $f_s\colon T_s \ra S$ through $s$ as in the statement of the claim. Then the thicknesses of the singularities of $C_{T_s}/T_s$ are bounded in absolute value by $m$ as $s$ runs over $V$ (note that the graphs of the central fibres are the same for all $s \in V$). By \cite[\S2]{Edixhoven1998On-Neron-models} the multidegrees of points in the closure of the unit section can be described purely in terms of the combinatorics of the dual graph. Since there are only finitely many possibilities for this dual graph, we find an integer $n$ such that for all $s \in V$, the composite 
\begin{equation*}
\Picn{n}_{C_{T_s}/T_s}\ra \frac{\Picz_{C_{T_s}/T_s}}{\bar{e}} = N_s
\end{equation*}
is surjective. 
\end{itemize}
\end{proof}
%
%component group of the quotient $\frac{\Picz_{C_{T_s}/T_s}}{\bar{e}}$ depends only on the 
%
%
%Combining this with the description of the closure of the unit section over a trait in \cite{Edixhoven1998On-Neron-models} 

%\begin{lemma}\label{lem:silly_graphs}
%Let $G$ be a graph, and $m$ an integer. Then there exists an integer $n$ with the following property: for every DVR $R$, for every generically smooth nodal curve $C/R$ with graph $G$ and all thicknesses of all singularities bounded by $m$, the composite $\Picn{n}_{C/R} \ra \Picz_{C/S}/\bar{e}$ is surjective. 
%\end{lemma}
%\begin{proof}
%Should be very easy from looking at Bas' paper. Maybe immediate, then no need for a lemma? 
%\end{proof}

\begin{lemma}\label{lem:trait_lemma}Let $R$ be an excellent noetherian domain and $x \in X \coloneqq \on{Spec}R$ be a point such that the localisation $R_x$ is a discrete valuation ring. Let $r \in R$ be a non-zero element. Then there exists an integer $m \in \bb{N}$ and an open subset $V \hra \overline{\{x\}}$ such that for all points $v \in V$, there exist a trait $f_v\colon T_v \ra X$ through $v$ with $\on{ord}_{T_v}f_v^*r \le m$. 
\end{lemma}
\begin{proof}
Since $X$ is regular at $x$ and the regular locus is open by excellence, we can assume after shrinking $X$ that $X$ itself is regular. Shrinking further we may assume that $r$ is a power of an element in $R$ which maps to a uniformiser in $R_x$, so it suffices to treat the case where $r$ itself maps to a uniformiser in $R_x$, and shrinking further we may assume $x = rR$. 

Choosing $V$ to be a small enough non-empty open of $\overline{\{x\}}$ we may assume that for all $v \in V$, we have $r \in \frak{m}_v$ and $r \notin \frak{m}_v^2$ (here we use that $x = rR$). Hence we can find a set of elements $a_1, \cdots, a_d \in \frak{m}_v$ such that the image of $r, a_1, \cdots, a_d$ in $\frak{m}_v/\frak{m}_v^2$ are a basis as a $\ca{O}_{X,v}/\frak{m}_v$-vector space. Then define $T_v$ to be the subscheme $V(a_1, \cdots, a_d)$ of $\on{Spec}\ca{O}_{X,v}$, and it is clear that $r$ pulls back to a uniformiser in the trait $T_v$. 
\end{proof}

\section{Consequences for torsion points}

\newcommand{\qq}{K}
Given a field $k$, by a \emph{variety over $k$} we mean a separated $k$-scheme of finite type. We fix a number field $\qq$ and an algebraic closure $\bar{\qq}$ of $\qq$. We write $\kappa(p)$ for the residue field of a point $p$. If $X/K$ is a variety and $d \in \bb{Z}_{\ge 1}$ then we write $X(\bar{K})^{\le d}$ for the set of $x \in X(\bar{K})$ with $[\kappa(x):K] \le d$. 

\begin{definition}
Let $S/K$ be a variety, and $A/S$ an abelian scheme. 
\begin{enumerate}
\item
We say \emph{the uniform boundedness conjecture holds for $A/S$} if for all $d \in \bb{Z}_{\ge 1}$ there exists $B \in \bb{Z}$ such that for all torsion points $p \in A(\bar{K})^{\le d}$, the point $p$ has order at most $B$. 
\item Given a section $\sigma \in A(S)$, we say \emph{the uniform boundedness conjecture holds for the pair $(A/S, \sigma)$} if for all $d \in \bb{Z}_{\ge 1}$ there exists $B \in \bb{Z}$ such that for all $p \in S(\bar{K})^{\le d}$, the point $\sigma(p)$ either has infinite order or has order at most $B$. 
\end{enumerate}
\end{definition}
If the uniform boundedness conjecture holds for $A/S$ then it clearly holds for the pair $(A/S, \sigma)$ for all $\sigma$. By considering the case of the universal PPAV, we deduce that if the uniform boundedness conjecture holds for all pairs $(A/S, \sigma)$ then the uniform boundedness conjecture itself (\ref{conj:UBC_reformulation_easy}) holds. 

\begin{lemma}\label{lem:bound_points_over_finite_field}%\check{Ask Bas about non-Commutative case}
Fix an integer $g \ge 0$ and a prime power $q$. Then there exists an integer $b = b(g,q)$ such that for every connected commutative finite-type group scheme $G/\bb{F}_q$ of dimension $g$ we have $\#G(\bb{F}_q) \le b$. 
\end{lemma}
\begin{proof}Since $\bb{F}_q$ is perfect, the scheme $G^{red}$ is a subgroupscheme and contains all the field-valued points, so we may assume $G$ is reduced and hence smooth. Again using that $\bb{F}_q$ is perfect, we can apply Chevalley's theorem to write an extension
\begin{equation*}
1 \ra T \times U \ra G \ra A \ra 1
\end{equation*}
where $A$ is abelian, $T$ is a torus and $U$ is connected and unipotent. We know $U$ is isomorphic (as a scheme over $\bb{F}_q$) to $\bb{A}^n_{\bb{F}_q}$ for some $n \le g$ by \cite[Remark A.3]{Kambayashi1974Unipotent-algeb}, so we have uniform bounds on the sizes of $A(\bb{F}_q)$, $T(\bb{F}_q)$ and $U(\bb{F}_q)$, from which the result is immediate. 
\end{proof}

\begin{theorem}
Let $S$ be a proper scheme over $K$, let $A/S$ be a finite-type commutative group scheme with connected geometric fibres, and let $\sigma \in A(S)$ be a section. Then for every integer $d >0$ there exists a bound $B$ such that for all $s \in S(\bar{K})^{\le d}$, the point $\sigma_s$ is either of infinite order or is torsion of order at most $B$. 
%Assume that $A/S$ is proper. Then the uniform boundedness conjecture holds for $(A/S, \sigma)$. 
\end{theorem}
\begin{proof}
To simplify the notation we treat the case $K=\bb{Q}$ and $d=1$; the general case is very similar. We begin by observing that we can `spread out' the proper scheme $S$, the group scheme $A$ and the section $\sigma$ over $\bb{Z}[1/N]$ for some sufficiently divisible integer $N$ - we use the same letters for the `spread out' objects. Let $p$ be a prime number not dividing $N$, and let $b = b(\on{dim}A/S, p)$ be the bound from \ref{lem:bound_points_over_finite_field}. Suppose we are given $s \in S(K)$ (with unique extension $\bar{s}\in S(\bb{Z}[1/N])$) such that $\sigma_s$ is torsion. For some $n$ the point $\tau \coloneqq p^n\sigma_s$ is torsion of order prime to $p$. Then the subgroupscheme of $A_{\bar{s}}$ generated by $\tau$ is \'etale over $\bb{Z}_p$, and so the order of the torsion point $\tau$ is bounded above by $b$ (independent of $s\in S(K)$). In this way we control the prime-to-$p$ part of the order, and by considering another prime $l$ we can control the whole order.  
\end{proof}

In the case of abelian varieties, we immediately obtain
\begin{corollary}\label{thm:UBC_for_good_compactifiation}
Let $S$ be a proper scheme over $K$, $U \sub S$ dense open, $A/U$ an abelian scheme, and $\sigma \in A(U)$ a section. Suppose there exists a finite-type commutative group scheme with connected fibres $\bar{A}/S$ extending $A$, such that $\sigma$ extends to $\bar{A}(S)$. Then the uniform boundedness conjecture holds for $(A/S, \sigma)$. 
\end{corollary}
By considering the diagonal section of tautological family of abelian varieties $A \times_S A$ over $A$, we deduce
\begin{corollary}\label{intro_thm:UBC_for_proper}
Suppose that $A/S$ is an abelian scheme with $S$ a \emph{proper} scheme over $K$. Then the uniform boundedness conjecture holds for $A/S$. 
\end{corollary}

Now \ref{intro_thm:UBC_for_proper} is rather trivial, and it is not a-priori clear how to construct interesting examples for \ref{thm:UBC_for_good_compactifiation}. In the case of jacobians of nodal curves the situation becomes much better; we understand exactly when N\'eron models exist by \cite{Holmes2014Neron-models-an}, and by \ref{thm:fin_type} we know that they are always of finite type, so any section of the jacobian family  extends to the identity component of the N\'eron model after taking some finite multiple (which is harmless for the arguments). We obtain
\begin{corollary}\label{final_cor}
Let $S$ be a proper scheme over $K$ and $C/S$ a regular family of nodal curves, smooth over some dense open $U \sub S$. Let $\sigma \in J(U)$ be a section of the jacobian family. If $J/U$ admits a N\'eron model over $S$ then the uniform boundedness conjecture holds for $(J/U, \sigma)$. 
\end{corollary}
\subsection{Examples}\label{examples}

We finish by giving some examples where this result can be applied. Recall from \cite{Holmes2014Neron-models-an} that the jacobian $J$ admits a N\'eron model if and only if the curve $C/S$ is \emph{aligned}; in other words that for all geometric points $s$ of $S$, the \emph{labelled dual graph} $\Gamma_s$ described in \ref{sec:NM_finite_type} has the property that

\begin{quote}
for every circuit $\gamma$ in $\Gamma_s$, and for every pair of edges $e_1$, $e_2$ appearing in $\gamma$, the labels of $e_1$ and $e_2$ satisfy a multiplicative relation of the form 
\begin{equation*}
\on{label}(e_1)^{n_1} = \on{label}(e_2)^{n_2}
\end{equation*}
for some positive integers $n_1$ and $n_2$. 
\end{quote}

Thus we see that $J$ admits a N\'eron model over $S$ (and so \ref{final_cor} applies) if any of the following hold (note that in each case $S$ must be proper in order to apply \ref{final_cor}). 
\begin{enumerate}
\item $S$ has dimension 1, so all labels are powers of a uniformiser in the base (this recovers a weaker version of a theorem of Silverman \cite{Silverman1983Heights-and-the} and Tate \cite{Tate1983Variation-of-th} on heights in families of varieties); 
\item $C/S$ is of compact type (all dual graphs are trees); then the N\'eron model is an abelian scheme (cf. \ref{intro_thm:UBC_for_proper}); 
\item \label{item:treelike}All dual graphs of $C/S$ are \emph{treelike} (i.e. a tree with some loops added). The N\'eron model need not be proper, but its component groups will be trivial; 
\item The complement of $U$ in $S$ (with reduced scheme structure) is smooth over $K$, cf. \ref{cor_intro:smooth_boundary}. Here we do not need to assume that $C$ is regular, since in this case we can resolve singularities without disturbing alignment. 
\end{enumerate}
In example \oref{item:treelike} the assumption that $C$ be regular is crucial, as we can illustrate with elliptic curves. Let $\mathcal{M}_{1,1}/K$ be the moduli stack\footnote{All our results work just as well when the base is an algebraic stack. } of stable 1-pointed curves of genus 1, and let $\mathcal{M}_{1,2}$ be the universal stable `pointed elliptic curve'. Now $\mathcal{M}_{1,1}$ has treelike-fibers so by example \oref{item:treelike} we can apply \ref{final_cor} to the regular curve $\mathcal{M}_{1,2}$ over $\mathcal{M}_{1,1}$ (we could alternatively use that $\mathcal{M}_{1,1}$ has dimension 1). However, if we want to re-prove the uniform boundedness conjecture for elliptic curves we would need to apply \ref{final_cor} to the universal stable curve $\mathcal{M}_{1,3}$ over $\mathcal{M}_{1,2}$. Note that $\mathcal{M}_{1,3}$ can be obtained by blowing up $X \coloneqq \mathcal{M}_{1,2} \times_{\mathcal{M}_{1,1}}\mathcal{M}_{1,2}$. Now $X /\mathcal{M}_{1,2}$ \emph{does} have treelike fibres, but when we resolve the singularities this breaks down, and indeed $\mathcal{M}_{1,3}$ over $\mathcal{M}_{1,2}$ does \emph{not} have treelike fibres and its jacobian does \emph{not} admit a N\'eron model. Similar considerations can be used to see that our results cannot be used to recover those of \cite{Cadoret2012Uniform-bounded}.

%Could talk about the conjectures being local for alterations of $S$, so giving more flexibility (and allowing to sort of reduce to semistable)? 

%show that this conjecture would follow from \ref{conjecture_ST} and a conjecture of Silverman \cite{?}. 
\bibliographystyle{alpha} %amsplain}
\bibliography{../../../prebib.bib}

\end{document}